\documentclass[reqno]{amsart}
\usepackage[T1]{fontenc}
\usepackage{lmodern,amsmath,amssymb,amsthm}
\usepackage{needspace}
\usepackage[expansion=false]{microtype}
\usepackage[colorlinks=true,linkcolor=blue,citecolor=blue,urlcolor=blue]{hyperref}
\hypersetup{pdftitle={Proof of the positive trace gap conjecture},pdfauthor={Nikolay Bogachev}}
\newtheorem{theorem}{Theorem}[section]
\newtheorem{lemma}[theorem]{Lemma}
\newtheorem{proposition}[theorem]{Proposition}

\newtheorem{corollary}[theorem]{Corollary}
\theoremstyle{remark}
\newtheorem{remark}[theorem]{Remark}
\DeclareMathOperator{\tr}{tr}
\DeclareMathOperator{\Ad}{Ad}
\DeclareMathOperator{\Trd}{Trd}
\DeclareMathOperator{\Nrd}{Nrd}
\DeclareMathOperator{\Comm}{Comm}
\DeclareMathOperator{\Gap}{Gap}
\DeclareMathOperator{\GL}{GL}
\DeclareMathOperator{\SL}{SL}
\DeclareMathOperator{\PSL}{PSL}
\DeclareMathOperator{\Isom}{Isom}

\newcommand{\R}{\mathbb R}
\newcommand{\C}{\mathbb C}
\newcommand{\Q}{\mathbb Q}
\newcommand{\Z}{\mathbb Z}
\newcommand{\F}{\mathbb F}
\newcommand{\HH}{\mathbb H}
\newcommand{\G}{\widetilde\Gamma}
\newcommand{\OO}{\mathcal O}
\title[]{Proof of the positive trace gap conjecture}
\author{Nikolay Bogachev}
\address{Department of Computer and Mathematical Sciences, University of Toronto Scarborough, 1095 Military Trail, Toronto, ON M1C 1A3, Canada}
\email{n.bogachev@utoronto.ca}

\begin{document}
\raggedbottom
\begin{abstract}
We prove that a lattice $\Gamma$ in $\PSL_2(\R)$ or $\PSL_2(\C)$ has positive trace gap, meaning that its traces are uniformly separated, if and only if it is derived from an admissible quaternion algebra. For cocompact Fuchsian groups, this proves the positive trace gap conjecture attributed to Sarnak by Geninska and Leuzinger in 2008. The same characterization by quaternion algebras holds if the difference set of traces is not dense.

Our method also gives a similar result for lattices in $\SL_d(\R)$, for every $d\ge3$: a lattice $\Gamma<\SL_d(\R)$ has positive trace gap if and only if all its traces are integers. Equivalently, after conjugation, it has finite index in the norm-one group of an order in a central simple algebra of degree $d$ over $\Q$ that splits over $\R$. We also discuss spectral consequences and other applications of the results and techniques.
\end{abstract}
\maketitle

\section{Introduction}

Let $\F$ be $\R$ or $\C$, let $G=\PSL_2(\F)$, and let $\Gamma<G$ be a lattice. Write $\G=\pi^{-1}(\Gamma)$ for the preimage of $\Gamma$ in $\SL_2(\F)$ where $\pi:\SL_2(\F)\to\PSL_2(\F)$ denotes the quotient map by $\{\pm I\}$. Set
$$
T(\Gamma)=\{\tr \gamma \mid \gamma\in\G\},
\qquad
\Gap(X)=\inf \{|x-y| : x,y\in X,\ x\ne y\}.
$$
Note that $T(\Gamma)=-T(\Gamma)$. The number $\Gap(T(\Gamma))$ will be referred to as the {\em trace gap} of $\Gamma$. 

\begin{theorem}\label{thm:main}
Let $\Gamma$ be a lattice in $\PSL_2(\F)$, where $\F=\R$ or $\C$. The following conditions are equivalent:
\begin{enumerate}
\item $\Gamma$ is derived from an admissible quaternion algebra;
\item $\Gap(T(\Gamma))>0$;
\item $\overline{T(\Gamma)-T(\Gamma)}\ne\F$.
\end{enumerate}
If $\Gamma$ is derived from an admissible quaternion algebra over a number field $k$ of degree $n=[k:\Q]$, then
$$
\Gap(T(\Gamma))\ge
\begin{cases}
4^{1-n},&\F=\R,\\
2^{2-n},&\F=\C.
\end{cases}
$$
\end{theorem}

Let us recall the quaternionic description of arithmetic lattices in these groups; see Maclachlan--Reid \cite{MR} and Takeuchi \cite{Takeuchi}. If $\F=\R$, let $k$ be a totally real number field and let $B/k$ be a quaternion algebra split at one distinguished real place and ramified at all other real places. If $\F=\C$, let $k$ have exactly one complex place and let $B/k$ be ramified at every real place. In either case the distinguished embedding $\sigma_0 : k \to \F$ identifies $B\otimes_{k,\sigma_0}\F$ with $M_2(\F)$. We call such a pair $(k,B)$ {\em admissible} and write $B^1=\{b\in B\mid\Nrd(b)=1\}$ for the quaternions of reduced norm one. For an order $\OO\subset B$, put $\OO^1=\OO\cap B^1$ and denote its image in $G = \PSL_2(\F)$ by $P\OO^1$. A lattice is \emph{arithmetic} if, after conjugation, it is commensurable with some $P\OO^1$, and it is said to be \emph{derived from a quaternion algebra} if, after conjugation, it is contained with finite index in some $P\OO^1$. Thus, the latter condition is stronger than arithmeticity.

The conjecture that the positive trace gap condition implies derivation from a quaternion algebra is attributed to Sarnak in Geninska--Leuzinger \cite[Conjecture 1.2(ii)]{GL}; see also Hao \cite[Conjecture 1.1(2)]{Hao23}. Sarnak's Schur lectures \cite{Sarnak} discuss the related \emph{bounded-clustering property}, which requires a uniform bound on the number of distinct traces in translates of a unit interval. More precisely, Sarnak \cite[footnote on page 211]{Sarnak} conjectured that a lattice
$\Gamma<\PSL_2(\R)$ is arithmetic if there exists a constant
$B=B(\Gamma)$ such that
$$
\#\bigl(T(\Gamma)\cap[L,L+1]\bigr)\le B \qquad\text{for every }L\in\Z.
$$
The analogous condition for Kleinian groups uses translates of a unit square in $\C$.

Luo and Sarnak \cite{LS} established the bounded-clustering property for arithmetic Fuchsian groups. The positive trace gap property for groups derived from an admissible quaternion algebra follows from the corresponding arithmetic bounds on traces; see \cite{LS,Sarnak}. Schmutz \cite{Schmutz} proposed a stronger arithmeticity conjecture, replacing bounded clustering by linear growth of the trace set. Namely, if $\Gamma<\PSL_2(\R)$ is a lattice and
$$
N_\Gamma(R):=\#\bigl(T(\Gamma)\cap[-R,R]\bigr)=O_\Gamma(R) \qquad (R\to\infty),
$$
then $\Gamma$ should be arithmetic. Bounded clustering implies this linear growth estimate, so Schmutz's conjecture implies Sarnak's bounded-clustering conjecture.

Schmutz~\cite{Schmutz} proposed a proof of the linear growth conjecture in the nonuniform case, but Geninska--Leuzinger~\cite{GL} identified a gap in the argument and proved Sarnak's bounded-clustering conjecture for nonuniform Fuchsian lattices. Hao \cite[Theorem A]{Hao23} recently proved both the bounded-clustering and positive-trace-gap characterizations for nonuniform lattices in $\PSL_2(\R)$ and $\PSL_2(\C)$. In a more recent work, Hao \cite[Theorem A]{Hao25} proved Schmutz's conjecture for nonuniform Fuchsian lattices, establishing arithmeticity under the weaker assumption $N_\Gamma(R)=o\bigl(R\log\log\log R\bigr).$

Related questions concerning traces and closed geodesics have also been studied for {\em semi-arithmetic} Fuchsian groups, introduced by Schmutz Schaller and Wolfart~\cite{SSW}. These groups are lattices that have a totally real invariant trace field and algebraic integer traces, but need not satisfy the boundedness condition at the other real embeddings in Takeuchi's arithmeticity criterion. Cosac and D\'oria \cite{CD} studied their systoles and constructed congruence towers with logarithmic systolic growth. Belolipetsky, Cosac, D\'oria, and Teixeira Paula \cite{BCDP} introduced the stretch, defined using equivariant Lipschitz maps, and proved that bounds on stretch, arithmetic dimension, and coarea leave only finitely many conjugacy classes of semi-arithmetic groups. Here the arithmetic dimension is the number of real places at which the invariant quaternion algebra splits.

More directly related to the questions considered here, Belolipetsky, Cosac, D\'oria, and Teixeira Paula \cite[Theorem 1.1]{BCDPmultiplicities} proved exponential growth of mean length multiplicities for semi-arithmetic Fuchsian groups of arithmetic dimension at most two admitting a modular embedding. Their result includes nonarithmetic groups. In the cocompact case of dimension two, they show that there exists $C_\Gamma>0$ such that $\#\bigl(T(\Gamma)\cap[R,R+1]\bigr) \le  C_\Gamma R^{1-\delta}$ for some $\delta=\delta(\Gamma)\in(0,1)$ and all $R\ge2$; see \cite[Lemma 3.2]{BCDPmultiplicities}. This gives $N_\Gamma(R)=O_\Gamma(R^{2-\delta})$, which is weaker than the linear-growth condition in Schmutz's conjecture. In particular, exponential growth of mean length multiplicities alone does not characterize arithmeticity. 

For arithmetic hyperbolic three-orbifolds, Marklof \cite{Marklof} studied the counting functions of distinct real and complex lengths and the resulting growth of mean length multiplicities. He obtained precise asymptotics for Bianchi groups and certain congruence subgroups, together with conditional asymptotics in the cocompact case.

The positive trace gap condition implies bounded clustering. Our arguments below do not prove the converse: density of the trace difference set produces close pairs of traces, which bounded clustering allows. The bounded-clustering conjecture for cocompact lattices remains open. We illustrate the distinction between the positive trace gap and bounded clustering conditions in the following remark. 

\begin{remark}
The zero trace gap corresponds to both arithmetic and nonarithmetic lattices. In Section \ref{sec:example}, we discuss an explicit example of an arithmetic lattice in $\PSL_2(\R)$ that is not derived from a quaternion algebra but whose trace set has zero gap. This example is obtained as the Fricke extension $\Gamma=\langle\Gamma_0(2),\pi(W)\rangle$ of the congruence group $\Gamma_0(2)$, where
$$
W=\frac1{\sqrt2}
\begin{pmatrix}0&-1\\2&0\end{pmatrix} \in\mathrm{SL}_2(\mathbb R). 
$$
Using work of Vinberg \cite{Vinberg} one can explicitly show that $\Gamma$ is not derived from a quaternion algebra. On the other hand, we prove that $\Gamma$ satisfies the bounded clustering property, $\Gap(T(\Gamma)) = 0$, and the trace difference set $T(\Gamma)-T(\Gamma) = 2\mathbb Z+\sqrt2\,\mathbb Z$ is dense in $\R$.
\end{remark}

The most involved implication of Theorem \ref{thm:main} is $(3) \implies (1)$; its proof is based on the combination of the following ideas. The important dynamical ingredient is the known fact (which we call the {\em double-coset lemma}) that a double coset $\Gamma a\Gamma$ is dense in $G$ whenever $a \in G$ does not commensurate the lattice $\Gamma < G$ (this fact is a consequence of the Ratner--Shah orbit closure theory). Then the key ingredient is purely algebraic: for $g\in\SL_2(\F)$ such that $\sqrt{\det(I+g)} \in \F$ is nonzero, we introduce the ``{\em square root}'' elements
$$
R(g) = \frac{I+g}{\sqrt{\det(I+g)}}
$$
and observe that $R(g)^2 = g$. For every $g$, this construction is well defined after replacing $g$ by $-g$ if necessary. These elements play a crucial role in several steps of our argument. First, we prove that if $g\in\G = \pi^{-1}(\Gamma) < \SL_2(\F)$ and $T(\Gamma)-T(\Gamma)$ is not dense in $\F$, then $\pi(R(g))\in\Comm_G(\Gamma)$ whenever $R(g)$ is defined. Then, under the same non-density assumption, we use a certain sequence $R(g_n)h^{-n}$, with $h\in\G$ a fixed loxodromic element, to show that $\Comm_G(\Gamma)$ is not discrete, and thus by Margulis's commensurator rigidity theorem we obtain that $\Gamma$ is arithmetic. The final step is to prove that $\Gamma$ is actually derived from a quaternion algebra. To that end, we use ideas and calculations of Vinberg \cite{Vinberg} applied to suitable elements $R(\pm g)$.

\medskip

Our methods also allow us to prove a similar statement for $\SL_d(\R)$ for every $d\ge3$. Here we use traces of the given matrices, without taking lifts. If $D$ is a central simple algebra over $\Q$, an order in $D$ is a subring containing $1$ which is a finitely generated $\Z$-module spanning $D$ over $\Q$. We write $D^1=\{x\in D \mid \Nrd(x)=1\}$ and $\OO^1=\OO\cap D^1$.

\begin{theorem}\label{thm:sld}
Let $d\ge3$, let $\Lambda$ be a lattice in $\SL_d(\R)$, and put $T_\Lambda=\{\tr g \mid g\in\Lambda\}$. The following conditions are equivalent:
\begin{enumerate}
\item After conjugation, $\Lambda$ is a subgroup of finite index in $\OO^1$, where $\OO$ is an order in a central simple algebra $D/\Q$ of degree $d$ and $D\otimes_\Q\R\cong M_d(\R)$;
\item $T_\Lambda\subset\Z$;
\item $\Gap(T_\Lambda)>0$;
\item $\overline{T_\Lambda-T_\Lambda}\ne\R$.
\end{enumerate}
In this case, $\Gap(T_\Lambda)\ge1$.
\end{theorem}

Let us turn our attention to the following spectral consequences of Theorem \ref{thm:main}. For a Fuchsian lattice $\Gamma$, let $\mathcal L(\Gamma)$ be the set of positive lengths of closed geodesics, including iterates and without multiplicities. Define
$$
\mathcal A(\Gamma)=\left\{a\in[0,\infty) \mid
\begin{array}{l}
\text{there exist }u_j<v_j\text{ in }\mathcal L(\Gamma),\ u_j\to\infty,\\ e^{v_j/2}(v_j-u_j)\to a \text{ as } j \to \infty  \end{array}\right\}.
$$

\begin{theorem}\label{thm:lengths}
Let $\Gamma<\PSL_2(\R)$ be a lattice.
\begin{enumerate}
\item The lattice $\Gamma$ is derived from a quaternion algebra if and only if there is a constant $c>0$ such that
\begin{equation}\label{eq:lengthsep}
|\ell-\ell'|\ge c\,e^{-\max\{\ell,\ell'\}/2} \quad \text{for all} \quad \ell,\ell'\in\mathcal L(\Gamma),\ \ell\ne\ell'.
\end{equation}
In this case one may take $c=2\Gap(T(\Gamma))$.
\item If $\Gamma$ is not derived from a quaternion algebra, then $\mathcal A(\Gamma)=[0,\infty)$. If it is derived, then $\mathcal A(\Gamma)$ is uniformly discrete and is contained in $[2\Gap(T(\Gamma)),\infty)$.
\end{enumerate}
\end{theorem}

\noindent Exponential separation of length spectra has been studied in a more general setting by Dolgopyat and Jakobson \cite{DJ}. Their Theorem~2.6 gives an exponential lower bound for groups with algebraic generators, with an exponent depending on the group: namely, they prove that if $X$ is a hyperbolic manifold such that the generators of $\pi_1 (X)$ belong to $\mathrm{PSO}_{n,1}(\overline{\Q})$ then there exist constants $c, \beta > 0$ such that $|\ell-\ell'|\ge c\,e^{-\beta\max\{\ell,\ell'\}}$ for all  $\ell,\ell'\in\mathcal L(\Gamma),\ \ell\ne\ell'.$ This holds true, in particular, for all finite-volume hyperbolic manifolds of dimension $n \ge 3$, see \cite[Corollary 2.7]{DJ}. Our Theorem~\ref{thm:lengths}(1) identifies all Fuchsian lattices for which the exponent $\beta = 1/2$ admits a positive uniform constant. We also note that Dolgopyat--Jakobson \cite[Theorem~3.1]{DJ} gives arbitrarily small gaps for a topologically generic set of hyperbolic surfaces. 

For a closed orientable hyperbolic surface $S=\Gamma\backslash\mathbb H^2$, write $\Gap(S)=\Gap(T(\Gamma))$, $\mathcal L(S)=\mathcal L(\Gamma)$, and $\mathcal A(S)=\mathcal A(\Gamma)$. Let $\mathcal M_g$ be the moduli space of such surfaces of genus $g$. Together with Borel's finiteness theorem \cite[Theorem 8.2]{Borel}, the preceding results give the following consequence.

\begin{corollary}\label{cor:moduli}
For every $g\ge2$, the set
$$
\mathcal E_g=\{S\in\mathcal M_g\mid \Gap(S)>0\}
$$
is finite. There is a constant $c_g>0$ such that every $S\in\mathcal E_g$ satisfies
$$
v-u\ge c_g e^{-v/2}
\qquad(u<v,\quad u,v\in\mathcal L(S)).
$$
Every $S\in\mathcal M_g\setminus\mathcal E_g$ has $\mathcal A(S)=[0,\infty)$.
\end{corollary}

For closed orientable hyperbolic surfaces, the Laplace spectrum determines the length set and hence the trace gap. In particular, the property of being derived from a quaternion algebra is preserved under Laplace isospectrality. This last assertion is classical; see Reid \cite[\S2 and Theorem 5.4]{Reid}. 

Finally, we observe that for hyperbolic lattices $\Gamma < \mathrm{Isom}^+(\HH^n)$, the double-coset lemma gives a metric criterion which allows to determine if a finite group $F < \Isom(\HH^n) $ belongs to the commensurator $\Comm_{\Isom(\HH^n)}(\Gamma)$. This is closely related to the recent work Belolipetsky et al \cite{BBKS} studying the behaviour of totally geodesic subspaces of hyperbolic orbifolds. We note, however, that our argument for the positive trace gap conjecture does not immediately extend to lattices in $\operatorname{Isom}(\mathbb H^n) = \mathrm{PO}_{n,1}(\R)$ for $n>3$, since the square-root construction of $R(g)$ relies on the degree-two Cayley--Hamilton theorem, whereas a scalar multiple of $I+g$ need not belong to the corresponding orthogonal group $\mathrm{O}_{n,1}$.

\subsection*{Organization of the paper} Section~\ref{sec:prelim} recalls the algebraic preliminaries, Vinberg's trace-field results, and the double-coset lemma. Section~\ref{sec:roots} introduces square roots $R(g)$ and proves Theorem~\ref{thm:main}, and Section~\ref{sec:sld} proves Theorem~\ref{thm:sld}. Section~\ref{sec:orbits} discusses a metric criterion for the commensurator of hyperbolic lattices. Section~\ref{sec:lengths} gives the consequences for length gaps, moduli spaces, and finite covers. Section~\ref{sec:example} gives the arithmetic example with zero trace gap.

\subsection*{Acknowledgements and AI use} I thank Misha Belolipetsky, Peter Sarnak, and Leone Slavich for their comments. 

I use large language models to discuss ideas, perform some computations, search the literature, and check the proofs. The paper is written entirely by me. 

After the manuscript was finished, in collaboration with Alexander Kolpakov we had an OpenAI model propose solutions to both Sarnak's and Schmutz’s conjectures in a guided Codex session, see \cite{BK}.

\section{Preliminaries}\label{sec:prelim}

\subsection{Central simple algebras and arithmetic lattices}\label{subsec:algebras}

Let $k$ be a field of characteristic zero. A \emph{central simple algebra} over $k$ is a finite-dimensional associative algebra $A$ with identity, center $k$, and no nonzero proper two-sided ideals. Its dimension is a perfect square; write $\dim_k A=d^2$, where $d$ is called its \emph{degree}. By the Artin--Wedderburn theorem, $A\cong M_r(\Delta)$ for a central division algebra $\Delta/k$. In particular, a central simple algebra need not be a division algebra. We say that $A$ \emph{splits} over an extension $K/k$ if $A\otimes_k K\cong M_d(K)$; see Maclachlan--Reid \cite[\S2.8]{MR}.

Choose a splitting field $K$ and an isomorphism $\iota:A\otimes_k K\cong M_d(K)$. The \emph{reduced trace} and \emph{reduced norm} are
$$
\Trd(x)=\tr\iota(x),\qquad \Nrd(x)=\det\iota(x).
$$
These functions take values in $k$ and do not depend on the choice of $K$ and $\iota$. Thus they agree with the ordinary trace and determinant in any matrix realization obtained by splitting $A$. The reduced trace is $k$-linear, and the reduced norm is multiplicative. We write
$$
A^1=\{x\in A^\times \mid \Nrd(x)=1\}.
$$

Two standard facts will be used below. The pairing $(x,y)\mapsto\Trd(xy)$ on $A$ is nondegenerate, as follows by extending scalars to a splitting field and using the matrix trace pairing. Also, every $k$-algebra automorphism of $A$ is the conjugation by an element of $A^\times$ due to the Skolem--Noether theorem; see \cite[\S2.9]{MR}.

When $k$ is a number field, let $\OO_k$ be its ring of integers. An \emph{order} in $A$ is a subring $\OO\subset A$ containing $1$ which is a finitely generated $\OO_k$-module spanning $A$ over $k$. The reduced characteristic polynomial of each element of $\OO$ has coefficients in $\OO_k$. In particular,
$$
\Trd(\OO)\subset\OO_k,\qquad
\Nrd(\OO)\subset\OO_k.
$$
We put $\OO^1=\OO\cap A^1$. Orders need not be maximal. For a fixed $A$, their norm-one groups are commensurable.

A \emph{quaternion algebra} over $k$ is a central simple algebra of degree two. It has a presentation
$$
B=k\oplus ki\oplus kj\oplus kij,
\qquad i^2=a,\quad j^2=b,\quad ij=-ji,
\qquad a,b\in k^\times.
$$
It is either a division algebra or isomorphic to $M_2(k)$. Its standard involution is given by $\overline{x}=\Trd(x)-x$, and satisfies
$$
x+\overline{x}=\Trd(x),\qquad
x\overline{x}=\Nrd(x),\qquad
x^2-\Trd(x)x+\Nrd(x)=0.
$$
At a real place $v$ of $k$, the algebra $B_v=B\otimes_k k_v$ is either $M_2(\R)$ or Hamilton's quaternion algebra $\mathbb H$. In the first case we say that $B$ is \emph{split} at $v$; in the second it is \emph{ramified}. Their norm-one groups are $\SL_2(\R)$ and the compact group of unit quaternions, respectively. At every complex place, $B_v\cong M_2(\C)$; see \cite[\S\S2.1, 2.5]{MR}.

The quaternionic constructions of arithmetic Fuchsian and Kleinian lattices use a number field $k$ and a quaternion algebra $B/k$ with the following properties:
\begin{enumerate}
\item For $\PSL_2(\R) \cong  \Isom^+ (\HH^2)$, the field $k$ is totally real and $B$ is split at one real place and ramified at all the others.
\item For $\PSL_2(\C) \cong  \Isom^+ (\HH^3)$, the field $k$ has exactly one complex place and $B$ is ramified at every real place.
\end{enumerate}
We call these pairs $(k,B)$ \emph{admissible}. A complex place means a conjugate pair of embeddings. The distinguished place gives an embedding $B\hookrightarrow M_2(\F)$, where $\F=\R$ or $\C$, and the projective image $P\OO^1$ is a lattice in $\PSL_2(\F)$. It is cocompact if and only if $B$ is a division algebra; see \cite[\S\S8.1--8.2]{MR}.

A lattice in $\PSL_2(\F)$ is \emph{arithmetic} if, after conjugation, it is commensurable with some $P\OO^1$ from this construction. It is \emph{derived from a quaternion algebra} if, after conjugation, it is contained with finite index in some $P\OO^1$. 

\medskip
For $\SL_d(\R)$, $d\ge3$, the \emph{inner} construction uses a central simple algebra $D/\Q$ of degree $d$ with $D\otimes_\Q\R\cong M_d(\R)$. An order $\OO\subset D$ gives an arithmetic lattice $\OO^1<\SL_d(\R)$, cocompact exactly when $D$ is a division algebra. Its traces are integers.

The \emph{outer} construction uses a quadratic extension $L/k$, with $k$ totally real, and a central simple algebra $E/L$ of degree $d$ with an involution $*$ of the second kind. This is an additive map satisfying
$$
(xy)^*=y^*x^*,\qquad (x^*)^*=x,
$$
whose restriction to $L$ is the nontrivial automorphism of $L/k$. The corresponding special unitary group has
$$
\mathbf{SU}(E,*)(k)=\{x\in E^\times \mid x^*x=1,\ \Nrd(x)=1\}.
$$
At one real place $v$, we require $L\otimes_{k,v}\R\cong\R\times\R$ and $E$ to split over both factors; this gives $\mathbf{SU}(E,*)(k_v)\cong\SL_d(\R)$. We also require compactness at all other archimedean places. For a $*$-invariant order $\OO\subset E$, projection of $\mathbf{SU}(E,*)(k)\cap\OO$ to the distinguished factor is then an arithmetic lattice.

Every arithmetic lattice in $\SL_d(\R)$ is, up to conjugation and commensurability, obtained from one of these two constructions; see \cite[Propositions 6.8.9, 6.8.14 and \S18.4]{MorrisAG}. Margulis's arithmeticity theorem applies to every lattice in $\SL_d(\R)$ for $d\ge3$ \cite[Theorem 5.2.1]{MorrisAG}. Our higher-rank theorem therefore has two further conclusions: it excludes the outer construction and places the entire lattice in the norm-one group of an order over $\Q$.

\subsection{Invariant trace fields and quaternion algebras}\label{subsec:vinberg}

Let $\Gamma<\PSL_2(\F)$ be a lattice, and let $\Gamma^{(2)}$ be the subgroup generated by squares of elements of $\Gamma$. We distinguish the ordinary trace field
$$
L_\Gamma=\Q(\tr g \mid g\in\G)
$$
from the invariant trace field
\begin{equation}\label{eq:invariantfield}
k_\Gamma=\Q((\tr g)^2 \mid g\in\G)=\Q(\tr h \mid h\in\widetilde{\Gamma^{(2)}}) \subset L_\Gamma.
\end{equation}
Here $\widetilde{\Gamma^{(2)}}$ denotes the preimage of $\Gamma^{(2)}$ under the quotient map $\pi: \SL_2(\F) \to \PSL_2(\F)$. The last equality is \cite[Theorem 1, Corollary 2, p.~180]{Vinberg}. Since
$$
\tr\Ad(g)=(\tr g)^2-1,
$$
the field $k_\Gamma$ is also generated by the adjoint traces. It is the smallest field of definition in the adjoint representation and is a commensurability invariant; see Vinberg \cite{Vin71} and \cite[p.~179, equations (1)--(2)]{Vinberg}.

The associated invariant quaternion algebra is
\begin{equation}\label{eq:invariantalgebra}
B_\Gamma=\operatorname{span}_{k_\Gamma}\widetilde{\Gamma^{(2)}}
\subset M_2(\F).
\end{equation}
It is a quaternion algebra over $k_\Gamma$, and $B_\Gamma\otimes_{k_\Gamma}\F\cong M_2(\F)$; see \cite[p.~186, equation (19)]{Vinberg} and \cite[Chapter 3]{MR}. 

These definitions do not require arithmeticity. In particular, $k_\Gamma$ need not be a number field. If $\Gamma$ is arithmetic, then $(k_\Gamma,B_\Gamma)$ is the admissible pair defining its commensurability class; see the proof of \cite[Theorem 4, pp.~187--189]{Vinberg} and \cite[Chapter 8]{MR}. We use this only after establishing arithmeticity.

\subsection{Density of double cosets and traces}\label{sec:double}

In this subsection, $G$ is a connected noncompact real linear Lie group with finite center and simple real Lie algebra, generated by one-parameter unipotent subgroups. The groups $\SL_d(\R)$ and $\SL_d(\C)$, $d\ge2$, and $\operatorname{Isom}^+(\mathbb H^n)$, $n\ge2$, have these properties. {\em Commensurability} of two subgroups of $G$ means that the intersection of two subgroups has finite index in both. The {\em commensurator} of $\Gamma < G$ in $G$ is the group
$$
\Comm_G(\Gamma)=\{a\in G\mid \Gamma \text{ and } a^{-1}\Gamma a
\text{ are commensurable}\}.
$$

The following lemma follows from Shah \cite[Corollary 1.5]{Shah}, applied to $a^{-1}\Gamma_1a$ and $\Gamma_2$, followed by left translation by $a$. Shah treats products of lattices in connected semisimple groups without compact factors, assuming that at least one lattice is irreducible. 

\begin{lemma}[Shah {\cite[Corollary 1.5]{Shah}}]\label{lem:double}
Let $\Gamma_1,\Gamma_2<G$ be lattices and let $a\in G$. If $\Gamma_2$ and $a^{-1}\Gamma_1a$ are commensurable, then $\Gamma_1a\Gamma_2$ is a finite union of right translates of $\Gamma_1$, and is closed and discrete. Otherwise $\Gamma_1a\Gamma_2$ is dense in $G$. In particular, $\Gamma a\Gamma$ is dense whenever $a\notin\Comm_G(\Gamma)$.
\end{lemma}

\begin{corollary}\label{cor:twisted}
Let $\Lambda<\SL_d(\F)$ be a lattice, where $d\ge2$ and $\F=\R$ or $\C$. If $A\in\GL_d(\F)$ does not commensurate $\Lambda$, then $\{\tr(\gamma A)\mid \gamma\in\Lambda\}$ is dense in $\F$.
\end{corollary}

\begin{proof}
Put $G=\SL_d(\F)$ and $\Lambda'=A\Lambda A^{-1}$. Since $\Lambda'$ is a lattice in $G$ and is not commensurable with $\Lambda$, Lemma~\ref{lem:double} implies that $\Lambda\Lambda'$ is dense in $G$. Right translation by $A$ therefore makes
$$
\Lambda A\Lambda=\Lambda\Lambda'A
$$
dense in $GA=\{B\in\GL_d(\F)\mid \det B=\det A\}$. The trace function maps this set onto $\F$: for any $z\in\F$, the matrix
$$
\begin{pmatrix}z-d+2&-\det A\\1&0\end{pmatrix}\oplus I_{d-2}
$$
has determinant $\det A$ and trace $z$. Finally, cyclic invariance of traces gives
$$
\{\tr(\gamma_1A\gamma_2)\mid \gamma_1,\gamma_2\in\Lambda\}
=\{\tr(\gamma A)\mid \gamma\in\Lambda\}.
$$
Continuity of the trace function proves the claim.
\end{proof}

For a lattice $\Gamma<\PSL_2(\F)$ and its preimage $\G < \SL_2(\F)$, a matrix $A\in\SL_2(\F)$ commensurates $\G$ if and only if $\pi(A)$ commensurates $\Gamma$. Thus the corollary applies to lattices in $\PSL_2(\F)$ without making a choice of lifts.

\section{Proof of Theorem \ref{thm:main}}\label{sec:roots}

\subsection{Square roots and arithmeticity}

In this section $G=\PSL_2(\F)$, as in the introduction.

\begin{lemma}\label{lem:root}
Let $g\in\SL_2(\F)$ and suppose that $\mu \in\F^\times$ satisfies $\mu^2=\tr g+2$. Then
$$
R(g)=\frac{I+g}{\mu}\in\SL_2(\F), \qquad R(g)^2=g,
$$
and, for every $g'\in\SL_2(\F)$,
\begin{equation}\label{eq:roottrace}
\mu\,\tr(g' R(g))=\tr(g' g)+\tr g'.
\end{equation}
Moreover, if $g\in\G$ and $\overline{T(\Gamma)-T(\Gamma)}\ne\F$, then $\pi(R(g))\in\Comm_G(\Gamma)$.
\end{lemma}

\begin{proof}
By the Cayley--Hamilton theorem, $g^2-(\tr g)g+I=0$, and therefore $(g+I)^2=(\tr g+2)g$. Together with $\det(g+I)=\tr g+2$, this shows that $R(g)$ has determinant one and satisfies $R(g)^2=g$. The trace identity \eqref{eq:roottrace} follows by linearity. When $g'\in\G$, its right-hand side belongs to the difference set, since
$$
\tr(g' g)-\tr(-g')\in T(\Gamma)-T(\Gamma).
$$
If $\pi(R(g))$ did not commensurate $\Gamma$, Corollary~\ref{cor:twisted} would make the left-hand side dense in $\F$, a contradiction.
\end{proof}

\begin{remark}\label{rem:sign}
For each $g\in\G$, the lemma applies after a possible change of sign. Over $\R$, choose $\varepsilon\in\{1,-1\}$ with $\tr(\varepsilon g)\ge0$; over $\C$, choose $\tr(\varepsilon g)\ne-2$. Both choices of $\mu$ give the same $\pi(R(g))$. This also applies to elliptic and parabolic elements.
\end{remark}

\begin{proposition}\label{prop:arith}
If $\overline{T(\Gamma)-T(\Gamma)}\ne\F$, then $\Gamma$ is arithmetic.
\end{proposition}

\begin{proof}
A lattice $\Gamma$ in $G$ is Zariski dense by Borel density \cite[Proposition 4.7.1]{Morris}, and contains a loxodromic element. After conjugation, choose a lift
$$
h=\begin{pmatrix}\lambda&0\\0&\lambda^{-1}\end{pmatrix}\in\G,
\qquad |\lambda|>1.
$$
In the real case choose the lift with $\lambda>1$. Zariski density also implies that there is an element
$$
M=\begin{pmatrix}a&b\\c&d\end{pmatrix}\in\G
\quad\text{with }ac\ne0.
$$
Over $\R$, replace $M$ by $-M$ if necessary so that $a>0$. Put $u_n=\lambda^{-2n}$ and
$$
g_n=h^nMh^n,\qquad
\tr g_n+2=\lambda^{2n}(a+2u_n+d u_n^2).
$$
Fix a square root $\alpha\in\F$ of $a$, positive in the real case. A local branch of the square root near $a\ne0$ gives, for all sufficiently large $n$, elements $\alpha_n$ satisfying
$$
\alpha_n^2=a+2u_n+d u_n^2,\qquad \alpha_n\longrightarrow\alpha.
$$
Thus $\mu_n=\lambda^n\alpha_n$ is a nonzero square root of $\tr g_n+2$, and Lemma~\ref{lem:root} gives $\pi(R(g_n))\in\Comm_G(\Gamma)$. Multiplying by $h^{-n}$, we obtain
\begin{equation}\label{eq:sequence}
S_n=R(g_n)h^{-n}
=\frac1{\alpha_n}
\begin{pmatrix}
a+u_n&b\\
cu_n&1+du_n
\end{pmatrix}
\longrightarrow
S=\begin{pmatrix}\alpha&b/\alpha\\0&\alpha^{-1}\end{pmatrix}.
\end{equation}
Every $S_n$ belongs to the preimage of $\Comm_G(\Gamma)$ in $\SL_2(\F)$, and infinitely many of these matrices are distinct, since their lower-left entries are nonzero and tend to zero. A discrete subgroup of a Lie group is closed and cannot contain a convergent sequence of infinitely many distinct elements. Hence the commensurator is nondiscrete. Margulis's commensurator rigidity theorem now gives arithmeticity, see \cite[\S1, p.~4]{Borel} and \cite[Chapter IX]{Margulis}. 
\end{proof}

\begin{corollary}\label{cor:dense}
If $\Gamma$ is nonarithmetic, choose $h,M,g_n$ as in the proof of Proposition~\ref{prop:arith}. For every sufficiently large $n$, the set
$$
\{\tr(\gamma g_n)+\tr\gamma \mid \gamma\in\G\}
$$
is dense in $\F$. In particular, $T(\Gamma)-T(\Gamma)$ is dense, and for every $\eta>0$ there are distinct $s,t\in T(\Gamma)$ with $|s-t|<\eta$.
\end{corollary}

\begin{proof}
The construction of \eqref{eq:sequence} does not require any condition on the trace set. Since $\Gamma$ is nonarithmetic, its commensurator is discrete. The matrices $S_n$ converge to $S$ and are different from $S$ for every $n$. Therefore only finitely many of them belong to the preimage of the commensurator. As $h$ belongs to $\G$, it follows that $\pi(R(g_n))$ is outside the commensurator for every sufficiently large $n$. Then it remains to apply Corollary~\ref{cor:twisted} and the identity \eqref{eq:roottrace}.
\end{proof}

\subsection{The invariant quaternion algebra}\label{subsec:containment}

We first work with the invariant algebra $B_\Gamma/k_\Gamma$ from Section~\ref{subsec:vinberg}, without assuming arithmeticity. The following calculation is Vinberg's description of its normalizer \cite[p.~186, equations (19)--(20)]{Vinberg}. We include the proof to keep track of determinant-one representatives.

\begin{lemma}\label{lem:rational}
Let $k=k_\Gamma$ and $B=B_\Gamma$. If $A\in\SL_2(\F)$ and $\pi(A)\in\Comm_G(\Gamma)$, then
$$
A=s b\quad\text{for some }s\in\F^\times,\ b\in B^\times.
$$
In particular,
\begin{equation}\label{eq:square}
A^2=\frac{b^2}{\Nrd(b)}\in B^1.
\end{equation}
\end{lemma}

\begin{proof}
The subgroups
$$
H=\Gamma\cap \pi(A)^{-1}\Gamma \pi(A) \quad\text{and}\quad \pi(A) H \pi(A)^{-1}
$$
both have finite index in $\Gamma$. Their invariant quaternion algebras are therefore both equal to $B$. Conjugation by $A$ takes the quaternion algebra of $H$ to that of $\pi(A) H \pi(A)^{-1}$ and fixes scalar matrices, so it defines a $k$-algebra automorphism of $B$. By the Skolem--Noether theorem, this automorphism is conjugation by some $b\in B^\times$. Thus $b^{-1}A$ centralizes $B$, and hence $B\otimes_k\F=M_2(\F)$, so it is a scalar $s\in\F^\times$.

Finally, $\det A=1$ gives $s^2\Nrd(b)=1$. Taking square of $A=sb$ proves \eqref{eq:square}, and the reduced norm of $b^2/\Nrd(b)$ is one.
\end{proof}

\begin{proposition}\label{prop:tracefield}
If $\overline{T(\Gamma)-T(\Gamma)}\ne\F$, then
$$
\G\subset B_\Gamma^1
\qquad\text{and}\qquad L_\Gamma=k_\Gamma.
$$
\end{proposition}

\begin{proof}
For $g\in\G$, choose $\varepsilon\in\{1,-1\}$ as in Remark~\ref{rem:sign}, so that $R(\varepsilon g)$ is defined. Its image $\pi(R(\varepsilon g))$ belongs to the commensurator by Lemma \ref{lem:root}, and Lemma~\ref{lem:rational} therefore gives
$$
\varepsilon g=R(\varepsilon g)^2\in B_\Gamma^1.
$$
Since $-I\in B_\Gamma^1$, this proves the first assertion of the current lemma. The ordinary trace of each $g\in\G$ is now its reduced trace in $B_\Gamma$, and hence belongs to $k_\Gamma$. Thus $L_\Gamma\subset k_\Gamma$, while the reverse inclusion follows from \eqref{eq:invariantfield}.
\end{proof}

\subsection{Proof of Theorem \ref{thm:main}}

\begin{proposition}\label{prop:order}
Suppose that $B/k$ is an admissible quaternion algebra, that $\OO_0\subset B$ is an order, that $\Gamma$ is commensurable with $P\OO_0^1$, and that $\G\subset B^1$. Then there is an order $\OO\subset B$ with $\G\subset\OO^1$. Consequently $\Gamma$ is derived from $B$.
\end{proposition}

\begin{proof}
For every $g\in\G$, some positive power $g^m$ belongs to $\OO_0^1$, since $\G\cap\OO_0^1$ has finite index in $\G$. Define polynomials in $\Z[X]$ by
$$
p_0=2,\qquad p_1(X)=X,\qquad p_{m+1}(X) = X p_m(X)-p_{m-1}(X).
$$
From the Cayley--Hamilton theorem, we have $g^{m+1} = (\tr g) g^m - g^{m-1}$. Taking traces and using induction yields $\tr(g^m)=p_m(\tr g)$, where every $p_m$ is monic for $m\ge1$. As $\tr g=\Trd(g)\in k$ and $\tr(g^m)\in\OO_k$, the monic polynomial
$$
p_m(X)-\tr(g^m)\in\OO_k[X]
$$
vanishes at $\tr g$. Hence $\tr g$ is integral over $\OO_k$ and belongs to $\OO_k$.

We now use Vinberg's order construction \cite[p.~188, equation (27)]{Vinberg}. Zariski density lets us choose $\gamma_1,\ldots,\gamma_4\in\G$ which are linearly independent over $\F$, and hence form a $k$-basis of $B$. The reduced trace pairing gives a $k$-linear isomorphism
$$
\Phi:B\longrightarrow k^4,\qquad
x\longmapsto(\Trd(x\gamma_1),\ldots,\Trd(x\gamma_4)).
$$
Since all traces of elements of $\G$ are in $\OO_k$, the module $\OO=\operatorname{span}_{\OO_k}\G$ is mapped by $\Phi$ into $\OO_k^4$. Since $\OO_k$ is Noetherian, this inclusion implies that $\OO$ is finitely generated. Moreover, $\OO$ spans $B$ over $k$, contains $1$, and is closed under multiplication because $\G$ is a group. It is therefore an order with $\G\subset\OO^1$, since $\Nrd(g) = \det g = 1$. Admissibility makes $P\OO^1$ a lattice, hence $\Gamma\subset P\OO^1$ has finite index.
\end{proof}

\begin{proof}[Proof of Theorem~\ref{thm:main}]
Condition (2) implies (3), since a positive gap makes $T(\Gamma)-T(\Gamma)$ disjoint from a punctured neighborhood of zero.

Assume (3). By Proposition~\ref{prop:arith}, the lattice $\Gamma$ is arithmetic. Thus its invariant pair $k=k_\Gamma$, $B=B_\Gamma$ is admissible, and $\Gamma$ is commensurable with $P\OO_0^1$ for an order $\OO_0\subset B$. Proposition~\ref{prop:tracefield} gives $\G\subset B^1$, and Proposition~\ref{prop:order} proves (1).

Finally, assume (1), and use the admissible field $k$ and order $\OO$ in that condition. Every trace lies in $\OO_k$. At any ramified real place $\sigma$ the group of norm-one quaternions is compact, and its reduced traces lie in $[-2,2]$. Thus, for distinct $s,t\in T(\Gamma)$ and $v=s-t$, one has $|\sigma(v)|\le4$ at every real place other than the distinguished place, if there is one.

Put $n=[k:\Q]$. In the real case the field $k$ is totally real, so
$$
1\le|N_{k/\Q}(v)|\le |v|\,4^{n-1}.
$$
In the complex case the distinguished pair of complex embeddings contributes $|v|^2$, and there are $n-2$ real embeddings. Therefore
$$
1\le|N_{k/\Q}(v)|\le |v|^2\,4^{n-2}.
$$
The left inequalities use that $v$ is a nonzero algebraic integer. These estimates show that every nonzero difference $v=s-t$ satisfies $|v|\ge4^{1-n}$ in the real case and $|v|\ge2^{2-n}$ in the complex case. This proves (2).
\end{proof}

\begin{remark}
Under the hypothesis of Proposition~\ref{prop:tracefield}, the relation with Vinberg's adjoint trace field can also be seen directly. If $R=R(\varepsilon g)$, then
$$
\tr\Ad(R)=(\tr R)^2-1=\varepsilon\tr g+1.
$$
Conjugation by an element of the commensurator acts $k_\Gamma$-linearly on the three-dimensional space $B_\Gamma^0=\ker\Trd$, so its adjoint trace belongs to $k_\Gamma$. The above identity therefore gives $\tr g=\varepsilon\bigl(\tr\Ad(R)-1\bigr)\in k_\Gamma$. Neither this observation nor Proposition~\ref{prop:tracefield} by itself proves that $k_\Gamma$ is a number field; arithmeticity is obtained by Proposition~\ref{prop:arith}.
\end{remark}

\begin{remark}\label{rem:lifts}
Let $b\in B^\times$. There exists $c\in k^\times$ such that $cb\in B^1$ if and only if $\Nrd(b)$ is a square in $k^\times$. Thus a matrix in $B^\times$ cannot always be multiplied by a scalar in $k$ to obtain determinant one.

In Lemma~\ref{lem:rational}, we write $A=sb$, where $A\in\SL_2(\F)$, $b\in B^\times$, and $s\in\F^\times$. The scalar $s$ need not belong to $k$, although the determinant condition gives $s^2=\Nrd(b)^{-1}\in k$. Nevertheless, $A^2=\frac{b^2}{\Nrd(b)}\in B^1.$ A related distinction appears in \cite[\S3.1.2 and Lemma 3.10]{BBKS}: if $Q$ is the matrix of a quadratic form over $k$, a matrix $M$ over $k$ may satisfy $M^{\mathsf T}QM=\mu Q$ for some $\mu\in k^\times$, while the normalized matrix $\mu^{-1/2}M$, which preserves the form, may only be defined over $k(\sqrt{\mu})$. Compare it also with the discussion of arithmetic lattices associated to skew-Hermitian forms in \cite[Remark 4.7]{BBKS}.
\end{remark}

\begin{remark}
The complex assertion concerns traces, not their absolute values. For example, the full trace set of $\PSL_2(\Z[i])$ is $\Z[i]$, so its complex gap is one. Its absolute trace values include $n$ and $\sqrt{n^2+1}$ for positive integers $n$, whose differences tend to zero.
\end{remark}

\section{Lattices in \texorpdfstring{$\SL_d(\R)$}{SLd(R)}}\label{sec:sld}

Throughout this section, $d\ge3$, $\Lambda<\SL_d(\R)$ is a lattice, and $T_\Lambda$ is its trace set. We take commensurators in $\GL_d(\R)$. This lets us use $g-I$ directly, regardless of the sign of its determinant.

\begin{lemma}\label{lem:minusone}
Suppose that $\overline{T_\Lambda-T_\Lambda}\ne\R$. If $g\in\Lambda$ and $\det(g-I)\ne0$, then $g-I$ commensurates $\Lambda$.
\end{lemma}

\begin{proof}
For every $\gamma\in\Lambda$,
$$
\tr(\gamma(g-I))=\tr(\gamma g)-\tr\gamma\in T_\Lambda-T_\Lambda.
$$
If $g-I$ did not commensurate $\Lambda$, Corollary~\ref{cor:twisted} would make these differences dense in $\R$.
\end{proof}

We will use the following version of Lemma~\ref{lem:rational}.

\begin{lemma}\label{lem:csarational}
Let $F\subset\R$ be a number field and let $E/F$ be a central simple algebra of degree $d$, embedded in $M_d(\R)$ so that $E\otimes_F\R=M_d(\R)$. Let $\Lambda_0<E^\times\cap\SL_d(\R)$ be Zariski dense. If $A\in\GL_d(\R)$ commensurates $\Lambda_0$, then $A=sb$ for some $s\in\R^\times$ and $b\in E^\times$.

Suppose also that $E$ has an involution $*$ whose restriction to $F$ is the nontrivial automorphism of a quadratic extension $F/k$, and that $h^*h=1$ for every $h\in\Lambda_0$. Then $b^*b\in k^\times$.
\end{lemma}

\begin{proof}
The group $H=\Lambda_0\cap A^{-1}\Lambda_0A$ has finite index in $\Lambda_0$ and is Zariski dense. Its real span is $M_d(\R)$, so it contains $d^2$ real-linearly independent elements, which form an $F$-basis of $E$. Since $AHA^{-1}\subset E$, conjugation by $A$ preserves $E$ and fixes its center pointwise. The Skolem--Noether theorem gives $b\in E^\times$ inducing the same automorphism. The matrix $b^{-1}A$ centralizes $E$, and hence $M_d(\R)$, so it is a real scalar.

Under the additional hypotheses, $bhb^{-1}=AhA^{-1}$ is unitary for every $h\in H$. Expanding $(bhb^{-1})^*(bhb^{-1})=1$ and using $h^*=h^{-1}$ gives
$$
h^{-1}(b^*b)h=b^*b.
$$
Since $H$ spans $E$, the element $b^*b$ lies in the center $F$; it is also fixed by $*$, and therefore belongs to $k^\times$.
\end{proof}

As we discussed in Section~\ref{subsec:algebras}, the classification of forms of type $A_{d-1}$ gives two arithmetic constructions of lattices in $\SL_d(\R)$: reduced norm one in a central simple algebra, and unitary groups for involutions of the second kind. These are called the inner and outer cases, respectively. Margulis's arithmeticity theorem applies to a lattice in $\SL_d(\R)$ because $d \ge 3$; see \cite[Theorem 5.2.1]{MorrisAG}. 

In the inner case the defining field is $\Q$. To see this, let $k$ be that field and consider the degree-$d$ algebra at its archimedean places. At a complex place its norm-one group is $\SL_d(\C)$. At a real place the algebra is $M_d(\R)$ or, when $d$ is even, $M_{d/2}(\mathbb H)$. Both norm-one groups are noncompact: in the quaternionic case $d\ge4$, and the matrices $\operatorname{diag}(e^t,e^{-t},1,\ldots,1)$ already form an unbounded subgroup. The arithmetic construction has only one noncompact archimedean factor, namely the given $\SL_d(\R)$. Thus $k$ has exactly one archimedean place, which is real, and $k=\Q$.

In the outer case there is a quadratic extension $L/k$, a central simple algebra $E/L$ of degree $d$, and an involution $*$ inducing the nontrivial automorphism of $L/k$. At the distinguished real place $v$ of $k$, we have $L\otimes_{k,v}\R=\R\times\R$, and $E$ splits over both factors. Projection to one factor identifies the real unitary group with $\SL_d(\R)$ and embeds $E$ in $M_d(\R)$. After conjugation, a finite-index subgroup $\Lambda_0<\Lambda$ is contained in
$$
\{g\in E^\times \mid g^*g=1,\ \Nrd(g)=1\}.
$$
All other archimedean factors are compact.

\begin{proposition}\label{prop:inner}
If $\overline{T_\Lambda-T_\Lambda}\ne\R$, then $\Lambda$ belongs to an inner arithmetic commensurability class.
\end{proposition}

\begin{proof}
Suppose instead that $\Lambda$ belongs to an outer class, and use $E$, $L/k$, $*$ and $\Lambda_0$ as above. Choose $g\in\Lambda_0$ such that
$$
\det(g-I)\ne0,\qquad g+g^{-1}\text{ is not scalar}.
$$
These conditions define a nonempty Zariski-open subset of $\SL_d$: the matrix $\operatorname{diag}(2,\ldots,2,2^{1-d})$ satisfies both. Borel density therefore gives such a $g$ in $\Lambda_0$.

By Lemma~\ref{lem:minusone}, $g-I$ commensurates $\Lambda$, and hence $\Lambda_0$. Lemma~\ref{lem:csarational} gives $g-I=tb$, where $t\in\R^\times$, $b\in E^\times$, and $b^*b\in k^\times$. Since $(g-I)b^{-1}=tI$ belongs to $E$ and is scalar, it belongs to the center of $E$. Thus $t\in L$, so applying the involution gives
$$
(g-I)^*(g-I)=t^*t\,b^*b\in k^\times.
$$
On the other hand, $g^*=g^{-1}$, and therefore
$$
(g-I)^*(g-I)=(g^{-1}-I)(g-I)=2I-g-g^{-1}.
$$
This is not scalar by our choice of $g$, a contradiction.
\end{proof}

\begin{proposition}\label{prop:sldcontainment}
Suppose that $\overline{T_\Lambda-T_\Lambda}\ne\R$. After conjugation, there is a central simple algebra $D/\Q$ of degree $d$ with $D\otimes_\Q\R=M_d(\R)$ such that $\Lambda\subset D^1$.
\end{proposition}

\begin{proof}
By Proposition~\ref{prop:inner} and the inner construction, we may assume that $\Lambda$ is commensurable with $\OO_0^1$ for an order $\OO_0\subset D$. Every element of $\Lambda$ commensurates $\OO_0^1$. Lemma~\ref{lem:csarational} therefore expresses every such element as a real scalar times an element of $D^\times$.

If $g\in\Lambda$ is nonscalar and $\det(g-I)\ne0$, then Lemma~\ref{lem:minusone} allows us to apply the same argument to $g-I$. We can therefore write
$$
g=sb,\qquad g-I=tc,
\qquad b,c\in D^\times,\quad s,t\in\R^\times.
$$
Hence
$$
c=\frac{s}{t}b-\frac1tI.
$$
Since $g$ is nonscalar, the elements $I,b$ are linearly independent over $\Q$ and can be extended to a $\Q$-basis of $D$. The resulting basis is also a real basis of $M_d(\R)$ under $D\otimes_\Q\R=M_d(\R)$. Comparing coordinates of $c\in D$ gives $s/t,1/t\in\Q$. Thus $s,t\in\Q$, so $g\in D$ and $\Nrd(g)=\det g=1$.

To prove the assertion for every element of $\Lambda$, let $U\subset\SL_d$ be the nonempty Zariski-open set of nonscalar matrices $x$ with $\det(x-I)\ne0$. For any $g\in\Lambda$, the conditions $h\in U$ and $gh^{-1}\in U$ define two nonempty open subsets of the irreducible algebraic variety $\SL_d$. Their intersection is nonempty and meets $\Lambda$ by Zariski density. The preceding argument puts both $h$ and $gh^{-1}$ in $D^1$, and therefore puts $g$ in $D^1$.
\end{proof}

\begin{proof}[Proof of Theorem~\ref{thm:sld}]
Condition (1) implies (2), since the reduced trace of an element of an order over $\Z$ is integral. The implications (2)$\Rightarrow$(3)$\Rightarrow$(4) are immediate. We prove (4)$\Rightarrow$(1).

By Proposition~\ref{prop:sldcontainment}, we may assume that $\Lambda\subset D^1$ and that $\Lambda$ is commensurable with $\OO_0^1$. For each $g\in\Lambda$, some positive power $g^r$ belongs to $\OO_0^1$, so its eigenvalues are algebraic integers. If $\alpha$ is an eigenvalue of $g$, then $\alpha^r$ is an eigenvalue of $g^r$. This implies that $\alpha$ is an algebraic integer. Thus $\tr g$ is an algebraic integer, and since $\tr g=\Trd(g)\in\Q$, we obtain $\tr g\in\Z$.

Choose $\gamma_1,\ldots,\gamma_{d^2}\in\Lambda$ forming a $\Q$-basis of $D$. The reduced trace pairing is nondegenerate, so $Q=(\Trd(\gamma_i\gamma_j))_{i,j=1}^{d^2}$ is an integral matrix with nonzero determinant $m$. If $g=\sum_i x_i\gamma_i\in\Lambda$, then every $\Trd(g\gamma_j)$ is an integer. Cramer's rule gives $x_i\in m^{-1}\Z$. Consequently $\OO=\operatorname{span}_{\Z}\Lambda$ is a full finitely generated $\Z$-module in $D$. As it contains $1$ and is closed under multiplication, it is an order.

Since $D\otimes_\Q\R=M_d(\R)$, the order $\OO$ is a lattice in the real vector space $M_d(\R)$, and $\OO^1$ is therefore discrete. It contains the lattice $\Lambda$, so we have $[\OO^1:\Lambda]<\infty$. This proves (1). The gap bound follows from integral traces.
\end{proof}

\section{A metric criterion for the commensurator}\label{sec:orbits}

The double-coset lemma also gives a geometric criterion for the commensurator of a lattice $\Gamma<\operatorname{Isom}^+(\mathbb H^n)$. A subset $A$ of a metric space $(X,d)$ is \emph{uniformly discrete} if there exists $\varepsilon>0$ such that $d(x,y)\ge\varepsilon$ for all distinct points $x,y\in A$. 

\begin{proposition}\label{prop:orbitsep}
Let $\Gamma<\operatorname{Isom}^+(\mathbb H^n)$ be a lattice, where $n\ge2$, and let $a\in\operatorname{Isom}(\mathbb H^n)$. The following are equivalent:
\begin{enumerate}
\item $a\in\Comm_{\operatorname{Isom}(\mathbb H^n)}(\Gamma)$;
\item $\Gamma o\cup a\Gamma o$ is uniformly discrete for some $o\in\mathbb H^n$;
\item $\Gamma o\cup a\Gamma o$ is uniformly discrete for every $o\in\mathbb H^n$.
\end{enumerate}
More precisely, for every $o\in\mathbb H^n$, the set of distances
$$
D_a(o)=\{d(\gamma o,a\delta o) \mid \gamma,\delta\in\Gamma\}
$$
is locally finite in $[0,\infty)$ if $a$ commensurates $\Gamma$, and is dense in $[0,\infty)$ otherwise.
\end{proposition}

\begin{proof}
Put $G=\operatorname{Isom}^+(\mathbb H^n)$ and $\Gamma'=a\Gamma a^{-1}<G$. If $a\in\Comm_{\operatorname{Isom}(\mathbb H^n)}(\Gamma)$, then $\Gamma a\Gamma$ is a finite union of translates of $\Gamma$, and is therefore closed and discrete. Otherwise $\Gamma$ and $\Gamma'$ are not commensurable, so Lemma~\ref{lem:double} implies that $\Gamma\Gamma'$ is dense in $G$. Right translation by $a$ then shows that $\Gamma a\Gamma=\Gamma\Gamma'a$ is dense in the component $Ga$. This argument applies whether or not $a$ preserves orientation.

Fix any point $o\in\mathbb H^n$. Note that $D_a(o)=\{d(o,bo)\mid b\in\Gamma a\Gamma\}.$ The function $b\mapsto d(o,bo)$ is continuous and maps either component of the isometry group onto $[0,\infty)$. Thus, density of the double coset gives density of $D_a(o)$. If instead the double coset is closed and discrete, its intersection with $\{b\in\operatorname{Isom}(\mathbb H^n) \mid d(o,bo)\le R\}$ is finite for every $R\ge0$, since this set is compact by properness of the action. Hence, $D_a(o)$ is locally finite.

The orbit $\Gamma o$ is uniformly discrete. Indeed, only finitely many elements of $\Gamma$ move $o$ a distance at most one, so their nonzero displacements, if any, have a positive minimum. The orbit $a\Gamma o$ has the same separation constant. If $a$ commensurates $\Gamma$, local finiteness of $D_a(o)$ also gives a positive lower bound for the nonzero cross-distances, so $\Gamma o\cup a\Gamma o$ is uniformly discrete. Otherwise density of $D_a(o)$ gives positive cross-distances tending to zero, which prevents uniform discreteness. Since $o$ was arbitrary, this proves all the assertions.
\end{proof}

\begin{corollary}\label{cor:finiteorbits}
Let $\Gamma<\operatorname{Isom}^+(\mathbb H^n)$ be a lattice, where $n\ge2$, and let $F<\operatorname{Isom}(\mathbb H^n)$ be a finite subgroup. The following are equivalent:
\begin{enumerate}
\item $F\subset\Comm_{\operatorname{Isom}(\mathbb H^n)}(\Gamma)$;
\item $\bigcup_{f\in F}f\Gamma o$ is uniformly discrete for some $o\in\mathbb H^n$;
\item $\bigcup_{f\in F}f\Gamma o$ is uniformly discrete for every $o\in\mathbb H^n$.
\end{enumerate}
\end{corollary}

\begin{proof}
Suppose that the union in (2) is uniformly discrete. Since the identity belongs to $F$, each pair $\Gamma o\cup f\Gamma o$ is uniformly discrete, and Proposition~\ref{prop:orbitsep} gives (1). Conversely, suppose that (1) holds and fix any $o\in\mathbb H^n$. For $f,f'\in F$, the element $f^{-1}f'$ lies in the commensurator. The proposition therefore implies that $\Gamma o\cup f^{-1}f'\Gamma o$ is uniformly discrete, and applying $f$ gives the same conclusion for $f\Gamma o\cup f'\Gamma o$. Taking the minimum of the separation constants for these finitely many pairs proves (3). Since (3) implies (2), the equivalence follows.
\end{proof}

The $fc$-subspaces of \cite{BBKS} are fixed-point sets of finite subgroups of $\Comm(\Gamma)$ and are used there to give a geometric characterization of arithmetic hyperbolic orbifolds. If such a subspace has dimension at least two, its stabilizer in $\Gamma$ acts on it with finite covolume \cite[Theorem 1.2]{BBKS}. Corollary~\ref{cor:finiteorbits} expresses the condition $F\subset\Comm(\Gamma)$ for a finite group $F$ in terms of uniform discreteness of $\bigcup_{f\in F}f\Gamma o$, for any $o\in\mathbb H^n$.

\begin{remark}
Uniform discreteness cannot be replaced here by bounded clustering. Indeed, for fixed $o\in\mathbb H^n$ and $R>0$, a packing argument gives a bound $N_R$ on the number of points of $\Gamma o$ in any ball of radius $R$. A union of $k$ isometric images of this orbit therefore has at most $kN_R$ points in any such ball. In particular, $\Gamma o\cup a\Gamma o$ satisfies this condition for every isometry $a$, whether or not $a$ commensurates $\Gamma$.
\end{remark}

\begin{remark}\label{rem:bisector}
Let $\Gamma<\operatorname{Isom}^+(\mathbb H^n)$ be a lattice, where $n\ge2$, and let $g,h\in\Gamma$ be loxodromic elements whose distinct axes $L_g,L_h$ meet at a point $o$. These axes project to intersecting closed geodesics in $\Gamma\backslash\mathbb H^n$. Suppose that their translation lengths $\ell(g)$ and $\ell(h)$ have irrational ratio. Let $r$ be the reflection in a hyperplane through $o$ which interchanges $L_g$ and $L_h$. In dimension two, this is reflection in either angle bisector. Then $r\notin\Comm_{\operatorname{Isom}(\mathbb H^n)}(\Gamma).$

To see this using Proposition~\ref{prop:orbitsep}, replace $h$ by $h^{-1}$ if necessary so that $r$ maps the positive direction of $L_g$ to the positive direction of $L_h$. By irrationality, there are positive integers $m_j,n_j$ such that $0<|m_j\ell(g)-n_j\ell(h)|\longrightarrow0.$ The points $rg^{m_j}o$ and $h^{n_j}o$ lie on $L_h$, at distances $m_j\ell(g)$ and $n_j\ell(h)$ from $o$ in the same direction. Hence
$$
0<d\bigl(rg^{m_j}o,h^{n_j}o\bigr)
=|m_j\ell(g)-n_j\ell(h)|\longrightarrow0.
$$
Thus $\Gamma o\cup r\Gamma o$ is not uniformly discrete, and the proposition gives the assertion. The intersection point $o$ is convenient for this calculation, but the conclusion of the proposition then shows that $\Gamma x\cup r\Gamma x$ is not uniformly discrete for any $x\in\mathbb H^n$.

There is also a direct proof from commensurability. If $r$ commensurated $\Gamma$, then $rg^kr^{-1}\in\Gamma$ for some positive integer $k$. This element and $h$ preserve $L_h$ and its orientation, and their translation lengths are $k\ell(g)$ and $\ell(h)$. The signed translations along $L_h$ arising from its orientation-preserving stabilizer in $\Gamma$ form a discrete subgroup of $\mathbb R$. Indeed, for every $R>0$, the isometries preserving $L_h$ and its orientation with absolute translation at most $R$ form a compact set, whose intersection with $\Gamma$ is finite. Consequently $k\ell(g)$ and $\ell(h)$ are integer multiples of the same positive number, contradicting irrationality. This argument only requires $\Gamma$ to be discrete. \end{remark}

\section{Length gaps and spectral consequences}\label{sec:lengths}

We retain the notation $\mathcal L(\Gamma)$ and $\mathcal A(\Gamma)$ from the introduction. The lengths in $\mathcal L(\Gamma)$ are not restricted to primitive closed geodesics. The identity $|\tr g|=2\cosh(\ell(g)/2)$ relates trace gaps to separation of distinct lengths.

\begin{proof}[Proof of Theorem~\ref{thm:lengths}(1)]
Write $f(x)=2\cosh(x/2)$, so $f'(x)=\sinh(x/2)$. Suppose that $\Gap(T(\Gamma))\ge\eta>0$. For distinct $\ell,\ell'\in\mathcal L(\Gamma)$, put $L=\max\{\ell,\ell'\}$. The mean value theorem gives
$$
\eta\le |f(\ell)-f(\ell')|
\le \sinh(L/2)|\ell-\ell'|
\le \tfrac12e^{L/2}|\ell-\ell'|.
$$
This proves \eqref{eq:lengthsep} with $c=2\eta$.

Conversely, assume \eqref{eq:lengthsep}. We first recall why $\mathcal L(\Gamma)$ has only finitely many values in each bounded interval. Truncating the cusps of the finite-area orbifold $S=\Gamma\backslash\mathbb H^2$ gives a compact subset $C\subset S$ meeting every closed geodesic, since no closed geodesic is contained in a cusp. Choose a compact set $K\subset\mathbb H^2$ whose image contains $C$, fix $o\in\mathbb H^2$, and put $D=\max_{x\in K}d(o,x)$. Every closed geodesic of length at most $R$ corresponds to a hyperbolic element $g\in\Gamma$ which, after conjugation in $\Gamma$, has its axis meeting $K$. For a point $x$ in this intersection,
$$
d(o,go)\le 2d(o,x)+d(x,gx)\le 2D+R.
$$
Since the discrete group $\Gamma$ acts properly discontinuously on $\mathbb H^2$, only finitely many elements of $\Gamma$ move $o$ a distance at most $2D+R$. Thus, there are only finitely many conjugacy classes of hyperbolic elements with translation length at most $R$. In particular, $\mathcal L(\Gamma)$ is locally finite and has a positive minimum $\ell_0$.

For $x<y$ in $\mathcal L(\Gamma)$ with $y-x\ge1$, we have
$$
f(y)-f(x)\ge f(\ell_0+1)-f(\ell_0)>0.
$$
If $y-x<1$, then \eqref{eq:lengthsep} gives
$$
f(y)-f(x)\ge \sinh(x/2)(y-x)
\ge \frac{c}{2}e^{-(y-x)/2}(1-e^{-x})
\ge \frac{c}{2}e^{-1/2}(1-e^{-\ell_0})>0.
$$
Both lower bounds are independent of $x$ and $y$. Denote their minimum by $\eta_h>0$. Distinct positive hyperbolic trace values therefore differ by at least $\eta_h$, and the same holds for distinct negative hyperbolic trace values.

The set $E=T(\Gamma)\cap[-2,2]$ is finite: elliptic elements are conjugate into the stabilizers of the finitely many cone points of $S$, and the remaining trace values in this interval are $\pm2$. Let $\eta_E>0$ be the smallest distance between distinct elements of $E$. Every positive hyperbolic trace is at least $f(\ell_0)>2$, and every negative hyperbolic trace is at most $-f(\ell_0)$. Consequently,
$$
\Gap(T(\Gamma))
\ge \min\{\eta_h,\eta_E,f(\ell_0)-2\}>0.
$$
Theorem~\ref{thm:main} now implies that $\Gamma$ is derived from a quaternion algebra.

Finally, suppose that $\Gamma$ is not derived from a quaternion algebra. By the equivalence just proved and Theorem~\ref{thm:main}, no positive constant $c$ satisfies \eqref{eq:lengthsep}. We can therefore choose $\ell_j<\ell'_j$ in $\mathcal L(\Gamma)$ such that $0<e^{\ell'_j/2}(\ell'_j-\ell_j)<\frac1j.$ Since there are only finitely many length values in every bounded interval, we must have $\ell'_j\to\infty$. The above inequality also gives $\ell'_j-\ell_j\to0$, and hence $\ell_j\to\infty$.
\end{proof}

\begin{proof}[Proof of Theorem~\ref{thm:lengths}(2)]
Put $f(x)=2\cosh(x/2)$, $P=f(\mathcal L(\Gamma))$, and $T=T(\Gamma)$. The set $P$ is locally finite, and $T\setminus(P\cup(-P))$ is finite, consisting of $\pm2$ and traces of elliptic or parabolic elements. Hence $T$ is locally finite.

Suppose first that $\Gamma$ is not derived from a quaternion algebra, so that $T-T$ is dense in $\R$. For every $M>0$, we claim that 
\begin{equation}\label{eq:taildensity}
\overline{(P\cap(M,\infty))-(P\cap(M,\infty))}=\R.
\end{equation}
To prove this, fix a bounded open interval $I$ and put $R=\sup_{z\in I}|z|$. A difference in $I$ involving a trace of absolute value at most $\max\{M,2\}$ has both traces with absolute values bounded by $\max\{M,2\}+R$. A difference in $I$ involving traces of opposite signs has both absolute values at most $R$, since they add to the absolute value of the difference. Local finiteness gives only finitely many difference values arising in these ways. Choose a nonempty open subinterval of $I$ avoiding them, and use density of $T-T$. The resulting difference is realized by two elements of $P$ or by two elements of $-P$, both of absolute value greater than $M$. In the latter case, changing signs and reversing their order realizes the same difference using elements of $P$. This proves \eqref{eq:taildensity}.

Fix $a\ge0$. By \eqref{eq:taildensity}, there are $p_j<q_j$ in $P$ such that $p_j\to\infty$ and $q_j-p_j\to a/2$; when $a=0$, choose strictly positive differences tending to zero. Set $u_j=f^{-1}(p_j)$ and $v_j=f^{-1}(q_j)$. Since $(f^{-1})'(t)=2/\sqrt{t^2-4}$, the mean value theorem gives $p_j<\xi_j<q_j$ with
$$
e^{v_j/2}(v_j-u_j)=2(q_j-p_j)\frac{e^{v_j/2}}{\sqrt{\xi_j^2-4}}.
$$
Since $q_j-p_j$ is bounded, $\xi_j/q_j\to1$, while $e^{v_j/2}/q_j\to1$ by the definition of $f$. The expression above therefore tends to $a$, proving the assertion for lattices which are not derived from a quaternion algebra.

Now suppose that $\Gamma$ is derived from a quaternion algebra over a totally real field $k$ of degree $n$, and put $D=P-P$. Every element of $P$ is an algebraic integer in $k$, with absolute value at most two at every embedding other than the given inclusion $k\subset\R$. For distinct $z,z'\in D$, write $z=p_1-p_2$ and $z'=p_3-p_4$, where $p_1,p_2,p_3,p_4\in P$. Then $z-z'$ is a nonzero algebraic integer in $k$, and for every other embedding $\sigma:k\to\R$,
$$
|\sigma(z-z')| =|\sigma(p_1)-\sigma(p_2)-\sigma(p_3)+\sigma(p_4)| \le \sum_{i=1}^4|\sigma(p_i)| \le 8.
$$
Taking the field norm gives
$$
1\le |N_{k/\Q}(z-z')| \le |z-z'|\,8^{n-1}.
$$
Hence $D$ is closed and uniformly discrete, with separation at least $8^{1-n}$.

If $u_j<v_j$ tend to infinity and $e^{v_j/2}(v_j-u_j)\to a<\infty$, then $v_j-u_j\to0$. Applying the mean value theorem to $f$ gives $f(v_j)-f(u_j)\to a/2.$ Indeed, for $u_j<\zeta_j<v_j$, we have $e^{-v_j/2}\sinh(\zeta_j/2)\to1/2$. The differences $f(v_j)-f(u_j)$ belong to $D$ and converge, so uniform discreteness forces them to be eventually constant. Since they are positive, we obtain
$$
\mathcal A(\Gamma)\subset 2(D\cap(0,\infty)).
$$
Every positive element of $D$ is at least $\Gap(T(\Gamma))$, so every element of $\mathcal A(\Gamma)$ is at least $2\Gap(T(\Gamma))$. Moreover, distinct elements of $\mathcal A(\Gamma)$ are separated by at least $2\cdot8^{1-n}$.
\end{proof}

\begin{proof}[Proof of Corollary~\ref{cor:moduli}]
By Theorem~\ref{thm:main}, every surface in $\mathcal E_g$ is arithmetic. Since each such surface has area $4\pi(g-1)$, Borel's finiteness theorem \cite[Theorem 8.2]{Borel} gives finitely many possibilities up to isometry. If $\mathcal E_g$ is nonempty, take $c_g=2\min_{S\in\mathcal E_g}\Gap(S)$ and apply Theorem~\ref{thm:lengths}(1); otherwise take any $c_g>0$. The last assertion is Theorem~\ref{thm:lengths}(2).
\end{proof}

\begin{remark}\label{cor:squarecover}
Let $S=\Gamma\backslash\mathbb H^2$ be a closed orientable hyperbolic surface, and let $S_2\to S$ be its mod-two homology cover. Its fundamental group is $\Gamma^{(2)}$, so the classical characterization of arithmeticity by $\Gamma^{(2)}$ being derived from a quaternion algebra \cite[Theorem 5.3]{Reid}, together with Theorem~\ref{thm:main}, gives
$$
S\text{ is arithmetic}\quad\Longleftrightarrow\quad\Gap(S_2)>0.
$$
Equivalently, $S$ is arithmetic if and only if some finite cover of $S$ has positive trace gap. 
\end{remark}

\section{An arithmetic example with zero trace gap}\label{sec:example}

Let $\pi:\mathrm{SL}_2(\mathbb R)\to\mathrm{PSL}_2(\mathbb R)$ be the quotient map by $\{\pm I\}$, where $I$ is the identity matrix. Consider the congruence subgroup
$$
\widetilde{\Gamma}_0(2)
=
\left\{
\begin{pmatrix}a&b\\2c&d\end{pmatrix}
\mid a,b,c,d\in\mathbb Z,\quad ad-2bc=1
\right\},
\qquad
\Gamma_0(2)=\pi(\widetilde{\Gamma}_0(2)).
$$
Define
$$
W=\frac1{\sqrt2}
\begin{pmatrix}0&-1\\2&0\end{pmatrix} \in\mathrm{SL}_2(\mathbb R),
\qquad \Gamma=\langle\Gamma_0(2),\pi(W)\rangle.
$$
Thus $\Gamma$ is the subgroup generated by $\Gamma_0(2)$ and the projective class of $W$. It is called the Fricke extension of $\Gamma_0(2)$; see \cite[\S1]{CMS}.

For $g=\begin{pmatrix}a&b\\2c&d\end{pmatrix}\in\widetilde{\Gamma}_0(2)$, direct computation gives
$$
WgW^{-1} =
\begin{pmatrix}
d&-c\\
-2b&a
\end{pmatrix} \in\widetilde{\Gamma}_0(2).
$$
Moreover, $W^2=-I$. Hence $W$ normalizes $\widetilde{\Gamma}_0(2)$, and the preimage $\widetilde{\Gamma}=\pi^{-1}(\Gamma)$ has the disjoint coset decomposition $\widetilde{\Gamma} = \widetilde{\Gamma}_0(2) \sqcup W\widetilde{\Gamma}_0(2).$ In particular, $[\Gamma:\Gamma_0(2)]=2$. Since $\Gamma_0(2)$ has finite index in $\mathrm{PSL}_2(\mathbb Z)$, the group $\Gamma$ is an arithmetic Fuchsian lattice: it is discrete, has finite covolume, and is commensurable with $\mathrm{PSL}_2(\mathbb Z)$.

Let us its full signed trace set $T(\Gamma)$. The equation $ad-2bc=1$ implies that $a$ and $d$ are odd, so every element of $\widetilde{\Gamma}_0(2)$ has even integral trace. Conversely, for every $r\in\mathbb Z$, the matrix
$$
A_r=
\begin{pmatrix}1&1\\2r-2&2r-1\end{pmatrix}
\in\widetilde{\Gamma}_0(2)
$$
has trace $2r$. Thus the trace set of this subgroup is exactly $2\mathbb Z$.

For the other coset, we have $\tr (Wg)=\sqrt2(b-c).$
Every value in $\sqrt2\,\mathbb Z$ occurs, since
$$
\operatorname{tr}\left(
W\begin{pmatrix}1&m\\0&1\end{pmatrix}
\right)=\sqrt2\,m
\qquad(m\in\mathbb Z).
$$

Consequently,
$$
T(\Gamma)=2\mathbb Z\cup\sqrt2\,\mathbb Z.
$$

To show that $\operatorname{Gap}(T(\Gamma))=0$, let the positive integers $m_j$ and $r_j$ be defined by $m_j+r_j\sqrt2=(3+2\sqrt2)^j$ for $j\ge1.$ Taking the product with the conjugate expression gives $m_j^2-2r_j^2=1$. Therefore the distinct trace values $\sqrt2\,m_j$ and $2r_j$ satisfy
$$
0<\sqrt2\,m_j-2r_j = \frac{\sqrt2}{m_j+\sqrt2\,r_j}
\longrightarrow0.
$$
Both traces exceed two, so these small gaps already occur among traces of hyperbolic elements.

The group nevertheless satisfies the bounded clustering property. Indeed, each interval of length one contains at most one point of $2\mathbb Z$ and at most one point of $\sqrt2\,\mathbb Z$. Thus, $\#\bigl(T(\Gamma)\cap[x,x+1]\bigr)\le2$ for all $x\in\mathbb R$.

We can also see directly that $\Gamma$ is not derived from a quaternion algebra. Its ordinary trace field and invariant trace field are, respectively,
$$
L_\Gamma =\mathbb Q(\operatorname{tr}(g) \mid g\in\widetilde{\Gamma}) =\mathbb Q(\sqrt2), \qquad k_\Gamma =\mathbb Q((\operatorname{tr}(g))^2 \mid g\in\widetilde{\Gamma}) = \mathbb Q.
$$
Since the two fields differ in this example, $\Gamma$ is arithmetic but not derived; see the proof of \cite[Theorem 4, p.~187]{Vinberg}.

Finally, its trace difference set is
$$
T(\Gamma)-T(\Gamma)=2\mathbb Z+\sqrt2\,\mathbb Z=\{2r+\sqrt2\,m \mid r,m\in\mathbb Z\},
$$
which is dense in $\mathbb R$. Thus, arithmeticity and bounded clustering are compatible with both zero trace gap and density of trace differences.

\bibliographystyle{siam}
\bibliography{biblio}
\end{document}